\documentclass[12pt]{amsart}

\usepackage{amsmath}
\usepackage{amssymb} 
\usepackage{array}
\usepackage[all]{xy}
\usepackage{enumitem}
\usepackage{hyperref}
\usepackage{verbatim} 
\usepackage{color}
\definecolor{battleshipgrey}{rgb}{0.52, 0.52, 0.51} 
\hypersetup{
       colorlinks=true,       
    linkcolor=battleshipgrey,     
    citecolor=battleshipgrey,        
    filecolor=magenta,      
    urlcolor=battleshipgrey          
}
\usepackage{pstricks-add}

\theoremstyle{plain}
\newtheorem{theorem}{Theorem}[section]
\newtheorem{lemma}[theorem]{Lemma}

\newtheorem{definition}[theorem]{Definition}
\theoremstyle{remark}
\newtheorem{remark}{Remark}[section]
\newtheorem{example}{Example}[section]
\newtheorem*{notation}{Notation}
\newtheorem*{acknowledgment}{Acknowledgment}

\numberwithin{equation}{section}

\newcommand{\K}{\mathbb{K}}

\newcommand{\R}{\mathbb{R}}

\newcommand{\PP}{\mathbb{P}}

\newcommand{\cP}{{\mathcal P}}

\newcommand{\cN}{\mathcal{N}}

\newcommand{\cA}{\mathcal{A}}
\newcommand{\cB}{\mathcal{B}}

\newcommand{\pG}{\mathsf{PG}}

\newcommand{\Gl}{\mathrm{Gl}}
\newcommand{\GL}{\mathrm{GL}}

\newcommand{\im}{\mathrm{im}}

\newcommand{\id}{\mathrm{id}}
\newcommand{\dom}{\mathrm{dom}}

\newcommand{\sett}[1]{{\{ #1 \}}}

\newcommand{\Bigsetof}[2]{\begin{Bmatrix} #1 \,\Big|\, #2 \end{Bmatrix}}

\newcommand{\msk}{\medskip}
\newcommand{\ssk}{\smallskip}
\newcommand{\nin}{\noindent}

\begin{document}

\title{A precise and general notion of  manifold}

\author{Wolfgang Bertram}

\address{Institut \'{E}lie Cartan de Lorraine \\
Universit\'{e} de Lorraine at Nancy, CNRS, INRIA \\
Boulevard des Aiguillettes, B.P. 239 \\
F-54506 Vand\oe{}uvre-l\`{e}s-Nancy, France}

\email{\url{wolfgang.bertram@univ-lorraine.fr}}



\begin{abstract} 
We give a completely formalized definition of a notion of  ``general manifold''.
It turns out that ``gluing data''  form an {\em equivalence-partially ordered set (e-pos)}, 
which is a special instance of an  {\em ordered groupoid}. We state and prove reconstruction theorems, allowing to reconstruct
general manifolds and their morphisms from such gluing data. 
To describe morphisms between manifolds, the notion of {\em natural relations between groupoids} is introduced, which
emphasizes the close analogy with natural transformations of general category theory. 
\end{abstract}

\maketitle

\begin{center}
{\sl
``La notion g\'en\'erale de vari\'et\'e est assez difficile \`a d\'efinir avec pr\'ecision.'' 

(Elie Cartan)}\footnote{
``It is rather difficult to define precisely the general notion of manifold.''
\cite{Ca28}, Chapitre III.}
\end{center}

\msk

\begin{center}
{\sl 
``Be either consequent or inconsequent, never both together.''

(An unknown moralist)}\footnote{
quoted from the preface to  \cite{FdV69}.}
\end{center}

\bigskip
\section{Introduction}

When talking about manifolds, physicists usually are consequent,
and mathematicians are not: physicists usually 
 do not try to give a formal definition of what a manifold 
is, 
 whereas mathematicians often start in a formal way, but, following the example of Elie Cartan,
  from a certain point on end phrases with ``and so on'', 
``whenever defined'', or similarly.  As far as I know, a  fully formalized 
definition of a sufficiently general notion of ``manifold'' has never been written up.\footnote{ To avoid misunderstandings, 
most modern texts, like e.g., \cite{Hu94, KoNo63}, are completely rigourous, of course; but authors make a deliberate choice not 
to push formalization up to the end, since they consider this to be unimportant or even counter-intuitive.}
Today, I will try to be consequent (to allow myself to be inconsequent tomorrow).
 The desire for consequence and full formalisation arose from a discussion with Anders Kock on Appendix D of my paper \cite{Be15a}: 
he remained sceptical about the general notion of ``primitive manifold (with atlas)'' defined there. Indeed, when working over
completely general base rings, where analytic constructions in charts may be replaced by ``formal'' constructions, 
 our intuition coming from real, finite-dimensional manifolds may fail, and one needs to
care about complete formalization. 
In  the present note I shall fill in, rather pedantically, the missing details of Appendix D loc.\ cit.
However, when doing so, I realized that this is not quite a simple exercise of re-writing:  some changes have to be made, in particular
with regard to {\em notation} --
namely, in loc.\ cit., I followed the classical notation $\phi_{ij}:V_{ji} \to V_{ij}$ (see e.g.\ \cite{Hu94}, or the 
\href{https://en.wikipedia.org/wiki/Differentiable_manifold#Atlases}{wikipedia-article}) for {\em transition maps}
between charts $(\phi_i,U_i)$ and $(\phi_j,U_j)$. 
But it turns out that this notation is cumbersome since it mixes up {\em two different} aspects of transition maps:
\begin{enumerate}
\item
the aspect of {\em restricting} things (``wherever defined'', i.e., on $U_{ij} = U_i \cap U_j$),
\item
the aspect of {\em transition between local coordinates} ($\phi_{ij}$ is a bijection).
\end{enumerate}

\nin
I think that, for better understanding the formal structure of the notion of manifold, it is advisable to seperate 
(1) and (2) notationally:
aspect (1) corresponds to a {\em partial order} $L$ on the index set $I$
(one chart may be included in another one, we write $i' \leq i$ or $(i',i) \in L$),
and aspect (2) corresponds to an {\em equivalence relation} $E$ on the index set $I$
(two charts are equivalent if their chart domains coincide, we write $(i,j) \in E$).
There is a natural compatibility condition turning the index set $I$ into what we call an ``e-pos''.  This is 
an interesting mathematical structure in its own right
(in fact, it is a special case of an {\em ordered groupooid}, cf.\ Appendix \ref{app:OG}): 

 \begin{definition}\label{def:epos}
An {\em equivalence-partially ordered set (e-pos)}
 is a set $I$ together with an equivalence relation $E$ and a partial order $L$ (or $\leq$)  on $I$ satisfying 
\begin{enumerate}
\item[{\rm (Epos)}] $\quad$
if $i' \leq i$ and   $(i,j) \in E$, then there exists a unique $j' \in I$ such that:  

$\, (i',j') \in E$ and $j' \leq j$.
\end{enumerate}
\end{definition}

\nin Adopting the convention that horizontal dashes mean ``in relation $E$'' and non-horizontal dashes ``in relation $L$ (with $i'$
placed lower than $i$)'', property (Epos) is represented by figures like these:
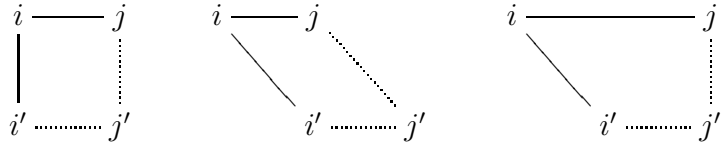
\begin{equation*}
\xymatrix{
    i \ar@{-}[d] \ar@{-}[r] & j \ar@{.}[d]   \\
    i' \ar@{.}[r]  & j' 
  } 
 \qquad 
\xymatrix{
    i \ar@{-}[rd] \ar@{-}[r] & j \ar@{.}[rd]  &  \\
   &  i' \ar@{.}[r]  & j' 
  } 
 \qquad 
\xymatrix{
    i \ar@{-}[rd] \ar@{-}[rr] & &  j \ar@{.}[d]    \\
   &  i' \ar@{.}[r]  & j' 
  } 
\end{equation*}
Every manifold $M$ with atlas $\cA$ has an underlying e-pos. 
If the atlas is finite, we get a finite e-pos. For instance,
the projective plane over a field $\K$ with its $3$ ``canonical'' charts $\phi_\nu$ with domains
$[x_0,x_1,x_2]$ where $x_\nu \not= 0$ for $\nu=0,1,2$, leads to 
an e-pos with $\vert I \vert = 12$:
each of the three canonical charts can be restricted to the intersection with the remaining $2$ charts domains
(giving rise to $6$ charts on the next lower level), and finally all three canonical charts can be restricted to
$U_0 \cap U_1 \cap U_2$.  This  e-pos is visualized as follows:
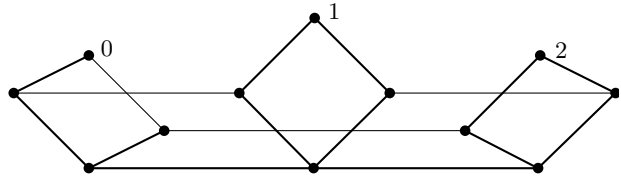
\begin{center}
\psset{xunit=0.5cm,yunit=0.5cm,algebraic=true,dotstyle=o,dotsize=3pt 0,linewidth=0.8pt,arrowsize=3pt 2,arrowinset=0.25}
\begin{pspicture*}(-1.3,-0.5)(15.08,4.8)
\psline(1,0)(12.94,0)
\psline[linewidth=0.4pt](3,1)(11,1)
\psline(3,1)(1,0)
\psline(11,1)(12.94,0)
\psline(6.97,0)(9,2)
\psline[linewidth=0.4pt](9,2)(15,2)
\psline(15,2)(12.94,0)
\psline(13,3)(15,2)
\psline(13,3)(11,1)
\psline(9,2)(7,4)
\psline(7,4)(5,2)
\psline(5,2)(6.97,0)
\psline[linewidth=0.4pt](5,2)(-1,2)
\psline(-1,2)(1,0)
\psline[linewidth=0.4pt](3,1)(1,3)
\psline(1,3)(-1,2)
\begin{scriptsize}
\psdots[dotsize=4pt 0,dotstyle=*](1,0)
\psdots[dotsize=4pt 0,dotstyle=*](12.94,0)
\psdots[dotsize=4pt 0,dotstyle=*](3,1)
\psdots[dotsize=4pt 0,dotstyle=*](11,1)
\psdots[dotsize=4pt 0,dotstyle=*](6.97,0)
\psdots[dotsize=4pt 0,dotstyle=*](9,2)
\psdots[dotsize=4pt 0,dotstyle=*](15,2)
\psdots[dotsize=4pt 0,dotstyle=*](13,3)
\rput[bl](13.4,2.92){$2$}
\psdots[dotsize=4pt 0,dotstyle=*](7,4)
\rput[bl](7.36,3.98){$1$}
\psdots[dotsize=4pt 0,dotstyle=*](5,2)
\psdots[dotsize=4pt 0,dotstyle=*](-1,2)
\psdots[dotsize=4pt 0,dotstyle=*](1,3)
\rput[bl](1.32,2.98){$0$}
\end{scriptsize}
\end{pspicture*}
\end{center}
On the other hand, {\em maximal atlases} give rise to very big, often uncountable, e-poses (Section \ref{sec:max}).
In the general case,
the transition functions $\phi_{ij}$ for every pair $(i,j) \in E$, together with restrictions for every pair
$(i',i) \in L$, define a
{\em morphism of the underlying e-pos into the pseudogroup of the model space $V$ of the manifold}
(Section \ref{sec:atlas}).
Conversely, 
every such morphism can be seen as ``gluing data'' of an abstract manifold $M$, that can be constructed by
taking some kind of quotient with respect to $E$, taking account of all restrictions by $L$ (Section \ref{sec:atlasdata}, 
Theorem \ref{th:reconstruct}).
For instance, the finite gluing data for the projective plane shown above satisfy our conditions when $\K$ is an
{\em alternative division algebra}, and hence the theorem yields a very natural construction of the octonion projective plane
${\mathbb{OP}}^2$ (see Example \ref{ex:projectiveplane2}).  
This equivalence also allows to perform constructions, like the one of {\em tangent bundles}, in complete generality:
whenever we have a functorial construction associating to the pseudogroup of $V$ a pseudogroup on another space $W$,
then the gluing data yield gluing data modelled on $W$, and thus give rise to a manifold modelled on $W$
(Theorem \ref{th:tangent}).

\ssk
On a next level, one wishes to describe {\em morphisms of manifolds} (i.e., smooth maps, if the manifold is smooth) via the gluing data.
 A formal analysis of the usual construction makes it clear that we won't get a plain
``morphism of e-poses'', since a map $f:M \to N$ won't induce a map assigning to a transition function on $M$ another one on $N$
(unless $f$ is a bijection). We rather get a {\em natural relation between e-poses}.
Indeed, this is a special instance of the more general notion of {\em natural relation between groupoids or small categories}
(Appendix \ref{app:NR}):  
natural relations are very closely related to natural transformations from general category; just like these, they give rise to 
double categories and $2$-categories (\cite{M98}), and thus 
may be of general interest for category theorists (cf.\ remark \ref{rk:NR}). 
We give some comments and mention some topics for further work in the final section \ref{sec:further}.

\begin{acknowledgment}
Besides Anders Kock for his critical and constructive remarks, I would like to thank the students of the lecture series
\href{http://iecl.univ-lorraine.fr/~Wolfgang.Bertram/WB-coursED.pdf}{``Concepts g\'eom\'etriques''}
 for their patience --
some of the material of this paper was presented there.
\end{acknowledgment}

\begin{notation}
By $\cP(X)$ we denote the power set of a set $X$.
{\em Relational composition} of binary relations $R \subset (C \times B), S \subset (B \times A)$
 is defined by
$R \circ S = \{ (c,a)\mid \exists b \in B : (c,b) \in R, (b,a) \in S \}$, and the {\em graph} of a map
$f: A \to B$ is $\Gamma_f = \{ (f(x),x) \mid x \in A \}$. This notation is best compatible with writing the function symbol $f$
on the left of its argument $x$.

Throughout, $V$ is a non-empty set,  
called the {\em model space},
and $I$ is a set, called the {\em index set}, which will serve as ``chart index'' for atlases. 
In practice, $V$ will often be equipped with a topology, whereas the set $I$ should be considered as ``discrete''.
\end{notation}

\section{From atlases to atlas data}\label{sec:atlas}

\begin{definition}\label{def:atlas}
A {\em chart of a set $M$, modelled on $V$} is a bijection $\phi:U \to W$ of a subset $U \subset M$ onto a subset
$W \subset V$. 
An {\em atlas on $M$} is a collection of charts
$\cA = (\phi_i:U_i \to V_i)_{i\in I}$, or $(U_i,\phi_i,V_i)_{i\in I}$, such   that:
\begin{enumerate}
\item
$[\phi_i =\phi_j] \Rightarrow [i=j]$ (atlas without repetion),
\item
$M = \cup_{i\in I} U_i$ (the atlas covers $M$),
\item (intersection and restriction of charts)
whenever $(U_i \cap U_j) \not= \emptyset$, then there exists a unique $k \in I$ (which we denote by $k=[ij]$) such that
$$
  U_k = (U_i \cap U_j), \quad \forall y \in U_k : \phi_k(y) =\phi_i(y)  \, .
$$
\end{enumerate}
Sometimes, one may prefer to work with the following slightly weaker version of (3):
\begin{enumerate}
\item[(3')]
for all $x \in (U_i \cap U_j)$,  there exists $k \in I$ such that
$$
x\in  U_k \subset (U_i \cap U_j), \quad \forall y \in U_k : \phi_k(y) =\phi_i(y)  \, .
$$
\end{enumerate}
If $V$ and $M$ are equipped with topologies, we generally assume that the sets $U_i$ and $V_i$ are open, and that
the $\phi_i$ are homeomorphisms. 
In this case, we will say that the atlas is {\em saturated} if {\em every} homeomorphism from a (non-empty) open subset of $M$ onto an
open subset of $V$ is of the form $\phi_k$ with some $k \in I$. 
\end{definition}

\nin
One may note that (3) implies (1), but (3') doesn't. 
See Section \ref{sec:max} for the definition of {\em atlases} with certain properties (such as smoothness), and of {\em maximal} such atlases.

\begin{definition}
For $(M,\cA)$ as above, we define 
\begin{enumerate}
\item[(a)]
a relation $E \subset (I \times I)$ by:  $(i,j) \in E$ iff $U_i = U_j$
(i.e., both charts have same domain, but need not coincide as charts),
\item[(b)]
a relation $L \subset (I \times I)$ by: $(i',i) \in L$ iff [$U_{i'} \subset U_i$, and 
$\forall y \in U_{i'}:
\phi_{i'}(y)=\phi_i(y)$]
(i.e., one chart is included in the other;  this implies $V_{i'} \subset V_i$).
\end{enumerate}
\end{definition}

\begin{lemma} Assume (1), (2), (3') hold. Then:
\begin{enumerate}
\item[(i)]
The relation $E$ is an equivalence relation on $I$, and $L$ is a partial order on $I$.
(We shall henceforth write $i' \leq i$ instead of $(i',i) \in L$.) 

Assume moreover that (3) holds. Then
\item[(ii)]
the triple $(I,E,L)$ forms an {\em e-pos} in the sense of Definition \ref{def:epos},
\item[(iii)]
if $i \leq m$ and $j \leq m$ and $(V_i \cap V_j ) \not= \emptyset$, 
then there exists a unique $k \in I$ with  $k \leq i$ and $k \leq j$ and $V_k = V_i \cap V_j$. 
\end{enumerate}
\end{lemma}

\nin We represent (ii) by a diagram, as in the Introduction, and (iii) as follows:
 $$
 
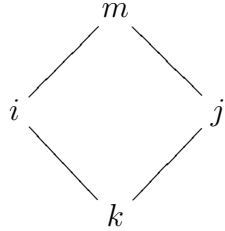
\xymatrix{
 & m \ar@{-}[ld] \ar@{-}[rd] \\
 i \ar@{-}[rd] & &  j \ar@{-}[ld] \\
 & k & }
  $$

\begin{proof}
It is clear that $E$ is an equivalence relation and $L$ a partial order.
Uniqueness in (ii) is clear since necessarily $U_{j'}=U_{i'}$ and $\phi_{i'}(x)=\phi_{j'}(x)$, and existence follows from (3):
$j' := [j,i']$ satisfies (Epos). 
Similarly, (iii) is proved: uniqueness follows from the given conditions, and existence by taking $k:=[ij]$.
\end{proof}

\begin{remark} \label{rk:uniqueness} 
By uniqueness, in every e-pos,  [ $i \leq j$ and $(i,j) \in E$ ] implies $i=j$.
Likewise, [ $i\leq m$, $j \leq m$, $(i,j) \in E$] implies $i=j$.
Thus in an e-pos, triangles such as the following are always degenerate in the sense that $i=j$:
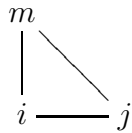
$$
\xymatrix{
    m \ar@{-}[d] \ar@{-}[rd] & \\
    i \ar@{-}[r]  & j
  } 
$$
\end{remark}

\begin{remark}
En e-pos is a special instance of an {\em ordered groupoid}, see Appendix \ref{app:OG}.
\end{remark}

\begin{remark}
In terms of relational composition, (Epos) implies that 
$E \circ L \subset L \circ E$.  This implies  that both 
$H_1 := L \circ E$ and $H_2 := E \circ L$ are transitive relations. 
In case of Definition 1.2,
$(i,k) \in H_\nu$ means, for $\nu = 1,2$,
``$U_i$ is included in $U_k$'', respectively, 
``$U_i$ is included in $U_k$, and $\phi_i$ extends to chart onto $U_k$''. Thus $H_1$ and $H_2$ are different relations, in general.
We will not work with them in this paper. They can be used to relate our approach to the one of V.V.\ Vagner who desrcribes
manifolds via {\em inverse semigroups}, cf.\ \cite{Sch79, LaS04}. 
\end{remark}

\begin{example}\label{ex:sphere}
The sphere $S^n$ with charts $s$, resp.\ $n$,
given by stereographic projection from the south pole  and from the north pole,
along with the resrictions $s_n$ and $n_s$ to the intersection domain, gives rise to the epos
$I = \{ n,s, n_s ,s_n \}$ with
$$
\xymatrix{
    s \ar@{-}[d]  & n \ar@{-}[d]  \\
    s_n \ar@{-}[r]  & n_s
  }
$$
Thus $E$ has three equivalence classes. Condition (Epos) is meaningless in this case.
\end{example}

\begin{example}\label{ex:projectiveplane1}
The e-pos of the projective plane over a (skew)field $\K$ has been described in the Introduction (Section 1): 
there are the $3$ ``canonical'' charts $\phi_\nu$, $\nu=0,1,2$, and $9$ other charts given by restricting them to 
all possible intersections of chart domains.
Every quadrangle appearing in this graph stands for a configuration given by three indices $i,j,k$, namely: if there are two horizontal edges,
the trapezoid stands for 
$\phi_j \vert_{U_k},\phi_k \vert_{U_j}$ (upper edge),
$\phi_j \vert_{U_k \cap U_i},\phi_k \vert_{U_j \cap U_i}$ (lower edge),
and if no edge is horizontal, the quadrangle stands for
$\phi_k$ (top vertex), $\phi_k\vert_{U_i}, \phi_k\vert_{U_j}$, $\phi_k \vert_{U_i \cap U_j}$ (bottom vertex; for esthetical reasons, we try to
represent such quadrangles by parallelograms, if possible). 
\end{example}

\begin{example}
For $n\geq 2$,  the e-pos of $\K \PP^n$ consists of $n+1$ different $n$-hypercubes 
of restrictions of the $n+1$ ``canonical'' charts to all possible intersections
of chart domains (so $\vert I \vert = (n+1) 2^n$); the vertices of different hypercubes are linked among each other if 
they correspond to the same intersection of domains. 
An elegant description of this e-pos is by identifying its index set $I$ with the {\em set of edges of an $n+1$-cube} (cf.\ \cite{Be15b} for notation):
the edge $(\beta,\alpha)$ with $\alpha = \beta \cup \{ i \}$ is identified with the restriction of the canonical chart $\phi_i$ to the intersection
$U_\beta = U_{\beta_0} \cap \ldots \cap U_{\beta_\ell}$ where $i \notin \beta$.
Two edges $(\beta,\alpha)$, $(\beta',\alpha')$ are equivalent iff $\alpha = \alpha'$.  They are in relation $L$,
$(\beta',\alpha') \leq (\beta,\alpha)$, iff [$\beta' \subset \beta$ and $i=i'$ (same ``direction'')].
\end{example}

\begin{definition}
If $(i,j) \in E$, we define the {\em transition map} 
$\phi_{ij}:= \phi_i \circ \phi_j^{-1}:V_j \to V_i$. 
\end{definition}

\begin{lemma} The transition maps satisfy, whenever $(i,j),(j,k) \in E$,
\begin{enumerate}
\item
$\phi_{ii}=\id_{V_i}$, 
\item
$\phi_{ij}\circ \phi_{jk}= \phi_{ik}$, 
\item
$\phi_{ij}^{-1}=\phi_{ji}$.
\end{enumerate}
\end{lemma}

\begin{proof} This follows directly from the definitions.
\end{proof}

\nin
The preceding definitions and lemmas can be summarized by saying that
$$
(I,E;L) \to (\cP(V), \cB_{loc}(V) ; \subset), \quad (k, (i,j);(i',i)) \mapsto (V_k , \phi_{ij}, (V_{i'}\subset V_i))
$$
is an (injective) morphism of ordered groupoids. This is what we are going to call ``atlas data'' in the following section.

\begin{definition}
Assume given a pseudogroup $G=(G_0,G_1) \subset (\cP(V),\cB_{loc}(V))$ of transformations of the model space $V$
(cf.\ Def.\ \ref{def:pseudogroup}), we 
say that  atlas data, and the atlas $\cA$, {\em are of type $G$}  if
$V_k \in G_0$ and $\phi_{ij} \in G_1$ whenever $k \in I$, $(i,j) \in E$. 
In case $(G_0,G_1)$ is the pseudogroup of locally defined diffeomorphisms of a topological $\K$-module $V$, we say that
{\em $\cA$ is a smooth atlas}. 
\end{definition}

\section{From atlas data to atlases}\label{sec:atlasdata}

\begin{definition}
We call {\em atlas data (with model space $V$ and chart index $I$)} the following:
the index set $I$ is an e-pos $(I,E,L)$,  and
\begin{enumerate}
\item
to each $i \in I$, is associated a  set $V_i \subset V$ (``chart range''), 
\item
to each pair $(i,j) \in E$ is associated a bijection  $\phi_{ij}:V_j \to V_i$ such that

$\qquad \qquad \phi_{ii}=\id_{V_i}$, $\qquad \phi_{ij}\circ \phi_{jk}= \phi_{ik}$, $\qquad \phi_{ij}^{-1}=\phi_{ji}$,
\item
if $i' \leq i$, then $V_{i'} \subset V_i$,
and if $(i,j) \in E, (i',j') \in E, i' \leq i, j' \leq j$, then
$$
\forall x \in V_{j'}: \quad \phi_{i'j'}(x) = \phi_{ij}(x) \, .
$$
\item
if $i \leq m$ and $j \leq m$ and $(V_i \cap V_j ) \not= \emptyset$, 
then there exists a unique $k \in I$ with  $k \leq i$ and $k \leq j$ and $V_k = V_i \cap V_j$. We write $[i,j]:= k$. 
\end{enumerate}
We say that atlas data are {\em topological} if
 $V$ carries a topology, all $V_i$ are open and all $\phi_{ij}$ are homeomorphisms.
\end{definition}

\nin
The idea how to reconstruct the manifold $M$ from these data is simple:
a point $x \in V_i$ shall be identified with a point $y \in V_j$ iff, possibly after restricting chart ranges to smaller sets $V_{i'}$, resp.\ $V_{j'}$,
there is a transition function such that $y = \phi_{j' i'} (x)$. That is, we shall define $M$ as a quotient of the set 
$$
S :=  \{ (x,i) \mid x \in V_i  \} \subset V \times I 
$$
under a suitable equivalence relation which arises from combining $E$ and $L$:

\begin{theorem}\label{th:reconstruct}
The following defines an  equivalence relation on $S$:  
\begin{itemize}
\item
$(x,i) \sim (y,j)$ iff:
\item
$\exists i' < i, \exists j' <j$: $(i',j') \in E$ and $(x,i') \in S, (y,j') \in S$, 
$\phi_{i'j'}(y)=x$: symbolically,
$$
\xymatrix{
    (x,i)  \ar@{-}[d]  & (y,j) \ar@{-}[d]  \\
    (x,i')  \ar@{-}[r]  & (y,j')
  }
$$
\end{itemize}
The quotient set $M:= S/\sim$ carries an atlas $\cA$ (satifying the condition (3') of Def.\ \ref{def:atlas}), with index set $I$, and
 charts defined by
$U_i :=  \{ [x,i] \mid x \in V_i \}$,
$\phi_i ([x,i]) := x$.
\end{theorem}

\begin{proof}
Symmetry and reflexivity of $\sim$ are clear.
Let us prove transitivity:
assume $(x,i) \sim (y,j)$ and $(y,j) \sim (z,k)$.
There exist
$i'\leq i$, $j'\leq j$, $j'' \leq j$, $k' \leq k$ such that
$(i',j'), (j'',k'') \in E$ and 
$\phi_{i'j'}(y)=x$, $\phi_{j'' k'}(z)=y$.
Using (4), we let
$m:= [j' j'']$.
By property (Epos) there exist $i''$ and $k''$ such that
$i'' \leq i'$, $(i'',m) \in E, k'' \leq k'$, $(k'',m) \in E$. 
By (3) and transitivity of $L$ and $E$,  it follows that $(x,i) \sim (z,k)$.
The whole argument is summarized by the following diagram:
$$
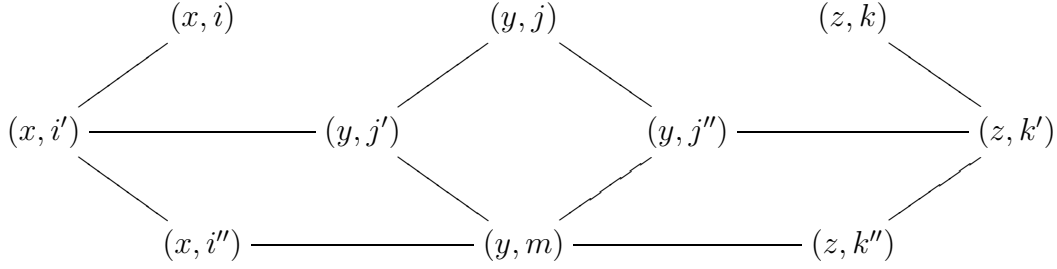
\xymatrix{
 & (x,i) \ar@{-}[ld] & & (y,j) \ar@{-}[ld]\ar@{-}[rd] & & (z,k)  \ar@{-}[rd] & \\
(x,i') \ar@{-}[rr] \ar@{-}[dr]& & (y,j') \ar@{-}[rd] & & (y,j'') \ar@{-}[rr]  \ar@{-}[dl]& & (z,k') \ar@{-}[ld] \\
 & (x,i'' )\ar@{-}[rr] & & (y,m) \ar@{-}[rr] & & (z,k'')  
}
$$

\ssk
Now let $M= S/\sim$ and define the $U_i$  as in the claim. Then each $[x,i] \in M$ belongs to some $U_i$, so the $U_i$ form a covering of $M$.
Let us  show that $\phi:U_i \to V_i$ is well-defined. 
When $U_i$ is defined as in the claim, the map
$V_i \to U_i$, $x \mapsto [x,i]$ is surjective. Let's show that it is injective:
assume $[x,i]=[y,i]$, so there exists $i' \leq i,j' \leq i$ with $\phi_{i'j'}(y)=x$.
But according to Remark \ref{rk:uniqueness}, this implies $i' = j'$, 
and so $x = \phi_{i'i'}(y)=y$. 
It follows that $V_i \to U_i$ is bijective, and hence its inverse
$\phi_i: U_i \to V_i$, $[x,i]\to x$ is well-defined.
Moreover, it follows that (using (3) for the last equality)
$$
\phi_j (\phi_i^{-1} (x)) = \phi_j [x,i] = 
\phi_j [\phi_{j'i'}(x),j] = \phi_{j'i'}(x) = \phi_{ji}(x) \, ,
$$
so the $\phi_{ij}$ indeed  describe the transition functions between the charts. 
This implies that, if $(i,j) \in E$, then $U_i = U_j$.

Let's show that (3') from Def.\ \ref{def:atlas} holds:
assume $[x,i]=[y,j] \in (U_i \cap U_j)$.
There exist $i' \leq i, j' \leq j$ such that $(i',j')\in E$ (whence $U_{i'}=U_{j'}$) and
$\phi_{i'j'}(y)=x$.
Thus (3') holds by taking $m=i'$. 
\end{proof}

\begin{remark}
We get ``almost'' an equivalence between atlasses and atlas data, the only difference being that, starting with (1), (2), (3) from
Def.\ \ref{def:atlas}, we only recover (1), (2), (3').
The reason for this is that, in the procedure of reconstrucing, we loose control over size of chart intersections (we juste require that they are 
non-empty). It would be a bit technical to impose conditions allowing  to keep such control, and we refrain from this. 
\end{remark}

\begin{theorem}
Assume that atlas data are topological and equip $M$ with the final topology with respect to all maps $\phi_i^{-1}:V_i \to M$
(so $U \subset M$ is open iff, for all $i \in I$, the set $\phi_i(U \cap U_i)$ is open in $V$). 
Then all charts $\phi_i:U_i \to V_i$ are homeomorphisms.
\end{theorem}

\begin{proof}
The maps $\phi_i^{-1}$ are continuous by definition of the topology on $M$.
To see that $\phi_j:U_j \to V_j$ is continuous,  let $W \subset V_j$ be open.
We have to show that, for all $i \in I$, the set
$Z:=\phi_i ( U_i \cap \phi_j^{-1} (W))$ is open in $V_j$.
When $(i,j) \in E$, then we have
$Z = \phi_{ij} (W)$, which is open since $\phi_{ij}$ is a homeomorphism.
Else, using  property (3') of Definition \ref{def:atlas}, restrict to smaller charts $i'$, $j'$ with $(i',j') \in E$, 
and the same argument implies that each point of $Z$ is an inner point, and so $Z$ is open.
\end{proof}

\nin
See \cite{BeNe05}, Theorem 5.3, for examples of smooth manifolds obtained by the construction described in  the  theorem.
Note that, already for usual, real manifolds, $M$ need not be Hausdorff, even if $V$ is Hausdorff. 
A simple counter-example is given by atlas data
$V_1 = V_2 = \R$, $V_{12} = \R^\times$, $\phi_{12}(x)=x$, so $M$ is ``$\R$ with origin doubled''.

\begin{theorem}\label{th:tangent}
Assume $M$ has an atlas $\cA$ of type $G=(G_0,G_1)$, and
assume that $T$ is a functor from $G$
to a pseudogroup $H=(H_0,H_1)$ acting on a space $W=TV$.
Then 
we may define a manifold $TM$ with atlas $T\cA$ given by all $(TV_k, T\phi_{ij})_{k\in I,(i,j)\in E}$, which is modelled on $W$ and of type $H$.
\end{theorem}

\begin{proof}
$T\cA = (TV_k, T\phi_{ij})_{k\in I,(i,j)\in E}$ are again atlas data, for the same e-pos $(I,E,L)$, and hence give rise, by the preceding theorem,
to an atlas.
\end{proof}

\nin
For instance, if all $\phi_{ij}$ are smooth, then $T$ may be the tangent functor, or any other Weil functor.
Thus one constructs the tangent bundle, or other Weil bundles, of a manifold $M$
(cf.\ \cite{BeS14}).

\section{Examples;  maximal atlases}\label{sec:max}

From an economical viewpoint, one is interested in keeping atlases of manifolds as small as possible. On the other hand, for 
theoretical purposes, most mathematicians are used to work with {\em maximal} atlases.

\begin{example}
An atlas is trivial, $I=\{ e \}$ and $\phi_{ee} = \id$, if and only if $M= U_e = V_e \subset V$, with just one chart. 
\end{example}

\begin{example}
An atlas has e-pos of the form given in Example \ref{ex:sphere} (sphere) iff $M$ arises from ``gluing together'' two subsets of $V$
along certain proper subsets which are identified via a bijection $\phi$.
\end{example}

\begin{example}\label{ex:projectiveplane2}
Consider the e-pos with 12 elements belonging to the projective plane (Introduction and Example
\ref{ex:projectiveplane1}), and let
$V = V_0 = V_1 = V_2 = \K^2$, and
\begin{align*}
V_{01}  = \{ (u,v) \mid u \not= 0 \} = V_{02},  \quad & V_{12} = \{ (u,v) \mid v \not= 0 \}, \quad
V_{012} = \{ (u,v) \mid u \not= 0, v\not= 0 \} .
\\
& \phi_{01}(u,v) = (u^{-1}, u^{-1} v ), 
\\
& \phi_{02}(u,v)  = (u^{-1}v,u^{-1}),
\\
& \phi_{12}(u,v)  = (v^{-1}u,v^{-1}) .
\end{align*}
By direct computation, the reader may check that each of these maps is of order two (so 
$\phi_{\nu \mu} = (\phi_{\mu \nu})^{-1} =\phi_{\mu \nu}$) iff in $\K$ the identities
$$
(x^{-1})^{-1} = x, \qquad x (x^{-1} y) = y
$$ 
are satisfied. Likewise, we have $\phi_{01} \circ \phi_{12}=\phi_{02}$ on $V_{012}$ iff, moreover in $\K$ we have: 
$$
(xy)^{-1} x = y^{-1} .
$$
Now, it is well-known that these identities hold in any alternative field (in particular, for $\K={\mathbb O}$, the octonions).
Thus our reconstruction theorem implies that for all alternative fields we may glue together copies of $\K^2$ to get 
a projective plane over $\K$.
(Essentially, this way of constructing the octonion plane is the one described by Aslaksen, \cite{As91}.)
If $\K$ is associative, similar formulas permet to describe the groupoid of transition functions. It is special for this example
that this groupoid embeds into a finite subgroup of the projective group $\PP\GL(n+1;\K)$.
\end{example}

\begin{example}\label{ex:full}
Assume $M=V$. The biggest possible atlas  is given by {\em all} possible local bijections of $V$, that is,
$I = \cB_{loc}(V)$, where an index $g$ is identified with the chart $g:\dom(g) \to \im(g)$ it describes.
 Then $L$ is given by inclusion: $f \leq g$ if $f$ is a restriction and corestriction of $g$, and
$$
E = \{ (f , g) \in \cB_{loc}(V)^2 \mid \, \dom(f)=\dom(g) \} 
$$
The transition functions are
$\phi_{fg}(x)= f g^{-1}(x)$. It follows that
the morphism of the e-pos $(I,E,F)$ to $(\cP(V),\cB_{loc} (V))$ is given by
$$
I \to \cP(V), \, f \mapsto \im(f), \qquad
E \to \cB_{loc}(V), \, (f,g) \mapsto f \circ g^{-1} .
$$
This atlas is {\em maximal} in the sense to be described next.
If $V$ carries a topology, then of course one will take only open sets and homeomorphisms.
\end{example}

\begin{example}\label{ex:lin}
If   $M=V$, we may also take $I={\rm Bij}(V)$, the group of all bijections (or some subgroup, like the group $\Gl(V)$ if $V$
is a linear space). 
Then $E = I \times I$, and $L$ is trivial.
Transition functions are as in the preceding example.
\end{example}

\begin{theorem}
Assume $\cA$ is a $G$-atlas on $M$, modelled on $V$, where $G$ is a pseudogroup of transformations on $V$.
Then there exists a {\em maximal $G$-atlas $\tilde A$ containing $\cA$}, which is defined by:
$$
\tilde \cA = \Bigsetof{ (U,\phi)}{ \begin{array} {c} \phi: U \to W \mbox{ bijection, } W \in G_0, U \subset M,   \\ 
\forall i \in I : [ \phi_i \circ \phi^{-1} : \phi(U \cap U_i) \to \phi_i(U \cap U_i) ] \in G_1 \end{array}
 }
$$
\end{theorem}

\begin{proof}
This statement is standard in many differential geometry textbooks (e.g.,
 \cite{KoNo63}, p.2).  The main point is 
  to check that $\phi \circ \psi^{-1} \in G_1$, for any two charts $\phi,\psi$ having same domain (or having non-empty intersection).
This is
locally  true, by intersecting with suitable charts from $\cA$, and the local-to-global property (PsG) of a pseudogroup (Def.\ \ref{def:pseudogroup})
permits to conclude that $\phi \circ \psi^{-1} \in G_1$.
We leave it to the reader to formalize these arguments (much like the proof of Theorem \ref{th:reconstruct}).
\end{proof}

\begin{definition}
A {\em $(V,G)$-manifold} is a set $M$  with a maximal $G$-atlas.
A  {\em smooth manifold (over $\K$)} is a $(V,G)$-manifold, where
 $V$ is a topological $\K$-module over a topological ring $\K$, and $G$ the pseudogroup of locally defined diffeomorphisms of $V$.
\end{definition}

\section{From morphisms to morphism data}

In this section, assume $(M,\cA, V)$ and $(M',\cA',V')$ are manifolds with atlas, and $f: M \to M'$ a map.
We describe $f$ with respect to the atlases.

\begin{definition}
Let $(I,E,L)$, resp.\ $(I',E',L')$ be the e-pos of $M$, resp.\ of $M'$.
Then $f$ induces  binary relations $F \subset (I \times I')$ and $R \subset (E \times E')$
by
\begin{align*}
F := & \bigl\{ (i',i) \in I \times I \mid  \,  f(U_i) \subset U_{i'} \bigr\} ,
\\
R := & \bigl\{ 
\bigl( (i,k), (i',k') \bigr) \in E \times E'  \mid \, 
(i',i),(k',k) \in F \bigr\} \, .
\end{align*}
\end{definition}

\nin
Elements $(i,k,i',k') \in R$ may be represented by a  parallelogram:
$$
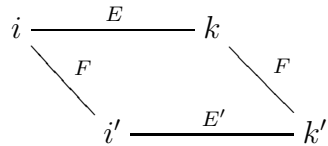
\xymatrix{ i \ar@{-} [rr]^E \ar@{-}[dr]^F  & &  k \ar@{-}[dr]^F \\
& i' \ar@{-}[rr]^{E'} & & k' }
$$


\begin{lemma}
The relation $F$  respects the partial order $L$ in the sense that:
\begin{enumerate}
\item
if $(i',i) \in F$ and $k \leq i$, then $(i',k) \in F$,
\item
if $(i',i) \in F$ and $i' \leq m'$, then $(m',i) \in F$.
\end{enumerate}
\end{lemma}

\begin{proof}
(1) $f(U_i) \subset U_{i'}$ and $U_k \subset U_i$ implies $f(U_k) \subset U_{i'}$.
(2): similar.
\end{proof}


\begin{definition}
Whenever $(i',i) \in F$, we define the {\em $(i',i)$-component of $f$} by
$$
f_{i' i} : =\phi_{i'} \circ f \circ \phi_i^{-1} :   V_i \to V_{i'}  .
$$
\end{definition}

\begin{lemma}
Whenever $\bigl( (i,k), (i',k') \bigr) \in R$, then 
$\phi_{k' i'} \circ f_{i' i} = f_{k' k} \circ \phi_{ki}$:
$$
\xymatrix{V_i \ar[r]^{f_{i' i}} \ar[d]^{\phi_{ik}} & V_i'  \ar[d]^{\phi_{k'i'}'} \\
V_k \ar[r]^{f_{k'k}} & V_k' }
$$
\end{lemma}

\begin{proof} 
Immediate from the definition of the $\phi_{ij}$ and $f_{j'j}$.
\end{proof}

If we want to reconstruct $f$ from the data $(f_{i'i})_{(i',i)\in F}$, then every $x$ must admit a neighborhood on which
$f$ is described by some component. This is not automatic -- it is a condition, very much like ``ordinary continuity'':

\begin{definition} 
We say that $f$ {\em is atlas-continuous} if, 
for all $x \in M$ and all $j \in I'$ with $f(x) \in V_j$, there exists $i\in I$ with $x \in V_i$ and
$f(V_i) \subset V_j$ (that is, $(j,i) \in F$). 
\end{definition}

\begin{remark}
If the atlas of $M$ is saturated, then this condition amounts to saying that $f$ is continuous in the usual sense.
\end{remark}

\section{From morphism data to morphisms}

We shall reconstruct $f$ from the ``morphism data'' $f_{ij}$. More precisely,
we shall see that {\em morphism data} are certain natural relations of certain  ordered groupoids (cf.\ Def.\ \ref{def:naturalrelation}).
Written out, this means:

\begin{definition}\label{def:morphdata}
Assume $(I,E,L,(\phi_{ij})_{(i,j) \in E})$ and $(I',E',L',\phi_{i'j'})_{(i',j') \in E'}$ are atlas data belonging to manifolds with atlases
$(M,\cA), (M',\cA')$. 
{\em
Morphism data} between them are given by:
a {\em special natural relation $(F,R)$} between the e-pos $(I,E,L)$ and $(I',E',L')$, 
that is, a pair of binary relations
$(F,R) \subset (\cP(I \times I') \times \cP(E \times E'))$ such that: whenever
$((i,k),(i',k')) \in R$, then $(i',i) \in F$ and $(k',k) \in F$,
and, for each
$(i',i) \in F$, a map $f_{i' i}:V_i \to V_{i'}$ such that, whenever
$\bigl( (i,k), (i',k') \bigr) \in R$, then
$$
\xymatrix{V_i \ar[r]^{f_{i' i}} \ar[d]^{\phi_{ik}} & V_i'  \ar[d]^{\phi_{k'i'}'} \\
V_k \ar[r]^{f_{k'k}} & V_k' }
\eqno (*)
$$
Moreover, the following {\em restriction} and {\em co-restriction properties} shall be satisfied:
\begin{enumerate}
\item
if $(i',i) \in F$ and $k \leq i$, then $(i',k) \in F$ and
$f_{i'k} = f_{i'i}\vert_{V_k}$,
\item
if $(i',i) \in F$ and $i' \leq m'$, then $(m',i) \in F$ and $\forall x \in V_i$:
$f_{i'i} (x)= f_{m'i}(x)$.
\end{enumerate}
Morphism data are called {\em full} if 

$\forall k \in I$, $\forall x \in U_k$: $\exists i \leq k$, $\exists i' \in I'$: $x \in U_i, (i',i) \in F$.
\end{definition}

\begin{theorem}\label{th:reconstruction2}
Let $(M,\cA),(M',\cA')$ be manifolds with atlas. Then,
given full morphism data, there is a unique map $f:M \to M'$ such that, whenever $x \in M$, $(i',i) \in F$ and $x \in U_i$,
$$
f(x) = \phi_{i'}^{-1} ( f_{i' i} (\phi_i(x))) \, . 
$$
\end{theorem}

\begin{proof}
Uniqueness is clear from the last formula, since (by the fullness condition) for every $x \in M$ there exist $(i',i) \in F$ with
$x \in U_i$. To prove existence, we define $f(x)$ by that formula, with respect to some choice of $(i',i) \in F$ with
$x \in U_i$, and we have to prove that with respect to another such choice, $(j',j)$,
$$
\phi_{i'}^{-1} ( f_{i'i} (\phi_i(x))) =
\phi_{j'}^{-1} ( f_{j'j} (\phi_j(x))) .
$$
If $(i,j) \in E$ and $(i',j') \in E'$, so 
$\bigl( (i,j), (i',j') \bigr) \in R$, then this follows from $\phi_{j',i'} \circ f_{i' i} = f_{j' j} \circ \phi_{ji}$.
If $(i,j)$ or $(i',j')$ are not in $E$, resp.\ $E'$, then using fullness and
 restriction and corestriction properties, we may find smaller charts with the corresponding property, and we get the same result.
 \end{proof}

\begin{example} Assume that $M=M'$.
Then $f_{ij}:=\phi_{ij}$ defines morphism data. The preceding theorem shows that these morphism data belong to the
identity map $\id_M$. Thus one may say that manifold data are ``morphism data of a would-be-identity''. 
\end{example}

\begin{definition}
Morphism data, and morphisms, are said to {\em have some property}, such as {\em smoothness}, if all components  $f_{ij}$ have this property. 
\end{definition}
 
 \begin{theorem}\label{th:comp}
If $g$ and $f$ are (smooth) composable morphisms (having some property, such as smoothness)
of manifolds with atlas, then so is $g \circ f$.
\end{theorem}

\begin{proof}
One has to check that $(g \circ f)_{\ell i} = g_{\ell k} \circ f_{ji}$, together with the relation $G \circ F$ from $I$ to $I''$, define
(smooth) morphism data.
No new ideas are involved here, and we may leave it to the reader.
\end{proof}

One may say that the proof of Theorem \ref{th:comp} consists of putting diagrams of the type (*) from Def.\ \ref{def:morphdata}
next to the other (horizontally), and thus define new diagrams.
But one may also put them one over the other (vertically), and this again defines new diagrams.
Indeed, this can be done for general natural relations (appendix \ref{app:NR}), and then defines a {\em double category}.
However, in the present case the second category structure is not relevant (since atlas data are a ``would-be-identity'' morphism,
as said above, and hence the second composition law somehow only reflects the composition of identity maps). 
Nevertheless, these remarks show that morphisms of manifolds sit inside very natural and bigger double categories. 
There is also a link with {\em $2$-categories}:

\begin{example}\label{ex:full2}
Consider the case $V= M$, $V'=M'$ with their atlases  described in example \ref{ex:full}.
Let $f:V \to V'$ be a map (continuous if data are topological). 
Then
$\phi_{gh}(x)=gh^{-1}(x)$ and 
$$
f_{k\ell}(x)= k f \ell^{-1}(x).
$$
Following the standard terminology from matrix theory, let us say that all components $f_{k\ell}$ are {\em equivalent} to each other.
In our example, $f$ can be recovered directly from its equivalent pictures, via
$f = f_{\id_{V'},\id_V}$. For a general manifold, there is no such formula, and one has to use the definition given in
Theorem \ref{th:reconstruction2}.
\end{example}

\begin{example}
Assume now that, with notation as in the preceding example, $V=V'$.
Then we may recover $f$ already from the restricted data $f_{kk} =  k f k^{-1}$.
Following usual terminology from the linear case, we say that these components are {\em similar} to each other.
\end{example}

\begin{definition}
Wih notation as in Definition \ref{def:morphdata} and Theorem \ref{th:reconstruction2},
assume that $M=M'$, $V=V'$, $\cA = \cA'$.
Then the collection $(f_{k'k})$ with $(k',k) \in F$ such that $k \leq k'$ are called {\em restricted morphism data}, or
{\em similarity data}.
\end{definition}

Morally, if $f$ is ``sufficiently close to the identity map of $M$'', then $f$ can be recovered from its restricted morphism data, as in Theorem
\ref{th:reconstruction2}, just as in linear algebra, where we describe {\em endo}morphisms by using the {\em same} base in the domain and in the
range space.



\section{Comments}\label{sec:further}

Some short remarks and comments on related topics and open problems:

\begin{enumerate}
\item
(Finite atlases.)
Like projective spaces (example \ref{ex:projectiveplane2}), many algebraic varieties come together with, more or less
``canonical'', finite atlases. In particular, this is true 
for {\em Grassmann and Lagrangian varieties}, and more generally, {\em symmetric $R$-spaces} (the ``Jordan geometries'' from
\cite{BeNe05, Be14}) where such atlases have a direct relation with Jordan theory. 
It should be interesting to describe and study the interaction between the abstract theory and the combinatorial and algebraic structure of
such atlases. 
\item
(Conceptual calculus.)
As said in the introduction, this is the starting point for the present work (\cite{Be15a}). 
The {\em first order difference groupoid} $M^\sett{1}$ of a manifold $M$ (cf.\ loc.\ cit.) is a natural example for the following item:
\item
($3$-categories.)
 As said above, $2$-categories are naturally related to the present approach. What about $3$-categories? 
 When the space $M$ carries itself the structure of some kind of groupoid or pregroupoid (e.g., principal bundles, Lie groupoids), then
such structure together with those described in the present work should be compatible, and thus give rise to (strict) higher order categories. 
\item
(Categorical aspects.)
For a discussion of purely categorial aspects of the notion of manifolds and their morphisms, see the $n$-lab, in particular
\url{https://ncatlab.org/nlab/show/manifold#morphisms_of_manifolds}.
Our definition of manifolds via e-poses  is also related to Lawson's construction of ordered groupoids via
{\em combinatorial groupoids}, \cite{La05}.
\end{enumerate}

\appendix

\section{Ordered groupoids}\label{app:OG}

\begin{notation} I use notation as in \cite{Be15a}, Appendix B:
a {\em groupoid} is an algebraic structure 
$
G = (G_0,G_1,\pi_1,\pi_0,\delta,\ast)
$,
 where $G_0$ is the set of {\em objects}, $G_1$ is the set of {\em morphisms},
$\pi_1:G_1\to G_0$ the {\em target}, and 
$\pi_0:G_1\to G_0$ the {\em source projection}, $\delta:G_0 \to G_1$ the {\em unit section}, and
$\ast : G_1 \times_{G_0} G_1 \to G_1$ the  composition  map.
Our convention is that $g \ast h$ is defined iff $\pi_1(h)=\pi_0(g)$.
The {\em inverse} of $g$ is denoted by $g^{-1}$.
A {\em small category (small cat)} is defined like a groupoid, without assuming existence of inverses.

\ssk
The definition of {\em ordered groupoid} goes back to Charles Ehresmann, see references in \cite{La05}.
The following form of the axioms is taken from \cite{AGM14}:
\end{notation}

\begin{definition} 
An  {\em ordered groupoid} is a groupoid $(G_0,G_1,\pi_0,\pi_1,\delta,\ast)$ together with
partial order relations $L$ (or $\leq $) on the sets $G_0$ and on $G_1$, such that:
\begin{enumerate}
\item[(OG0)] $\forall x,y \in G_0$: $x \leq y$ iff $\delta(x) \leq \delta(y)$,
\item[(OG1)]  $g \leq h$ iff $g^{-1} \leq h^{-1}$,
\item[(OG2)] if $g\leq h$, $g' \leq h'$ and if $g \ast g'$ and $h\ast h'$ are defined, then
$g \ast g' \leq h \ast h'$,
\item[(OG3)]
for all morphisms $g:x \to y$ and all objects $x'$ with $x' \leq x$, there exists  a
 unique morphism $g':x' \to y'$ with $g' \leq g, y'  \leq y$:
$$
\xymatrix{
    x \ar@{-}[d]^L  \ar[r]^g & y \ar@{.}[d]^L    \\
    x' \ar@{.>}[r]^{g'}  & y' 
  } 
  $$
 \end{enumerate}
\end{definition}

\nin  One may think of $g'$ as a kind of {\em restriction of $g$ to $x'$}.  

\begin{remark}
Every partially ordered set gives rise to a small category (subcat of the pair groupoid), and hence the set 
$[x] = \{ x' \mid x' \leq x \}$ is the object set of a small cat. Every $g \in G_1$ then defines a functor from
$[\pi_0(g)]$ to $[\pi_1(g)]$.
\end{remark} 



\begin{example}(e-poses)
Recall that an equivalence relation $E$ on a set $I$ is the morphism set of a groupoid with object set $I$.
Thus an e-pos $(I,E,L)$ (Definition \ref{def:epos}) is an ordered groupoid. Here the partial order on $E$ is completely determined by the
partial order on $I$: necessarily, $(i',j') \leq (i,j)$ iff [$i'\leq i$ and $j' \leq j$].
\end{example}

\begin{example}(local bijections)
Recall that {\em endorelations on a set $V$} form a small category $(C_0,C_1)=(\cP(V),\cP(V \times V))$.
Morphisms are relations $R \subset (V \times V)$ with  composition being relational composition, 
source and target
\begin{align*}
\pi_0(R) = \dom (R) & = \{ x \in V \mid \, \exists y \in V : (y,x) \in R \} \\
\pi_1(R) = \im(R) & = \{ y \in V \mid \exists x \in V : (y,x) \in R \} .
\end{align*}
Objects and morphisms carry a natural partial order $L$  given by inclusion of sets.
The 
{\em full pseudogroup of $V$}, or {\em groupoid of local bisections of the pair groupoid of $V$},
is the subcat $(\cP(V), \cB_{loc}(V))$
 of $(C_0,C_1)$ whose morphisms are precisely the {\em local bijections}, that is, 
 the {\em injective and locally functional} relations $R$:
$$
(y,x) ,(y',x),(y,x') \in R \qquad \Rightarrow \qquad y=y', x=x' .
$$
These are  the graphs $\Gamma_f$ of bijections $f :V' \to V''$ with
$V',V'' \subset V$, and $f \leq g$ means that $f$ is a restriction of $g$.
Properties (OG0) -- (OG3) are easily checked.
\end{example}

\begin{example} (pseudogroups of smooth maps)
Assume now that $V$ carries additional structure, e.g., $V$ is a topological vector space over a topological field $\K$.
Then we may define the subgroupoid of $\cB_{loc}(V)$ of all locally defined diffeomorphisms of class $C^k$:
its objects are open subsets of $V$, and its morphisms graphs of $C^k$-diffeomorphisms 
$f:V' \to V''$ with $V',V''$ open in $V$.
It is again an ordered groupoid, since restriction (together with co-restriction to the image) of $f$ is again a
$C^k$-diffeomorphism.
This defines an ordered groupoid $\cB_{loc}^k(V)$, called the {\em pseudogroup of local $C^k$-diffeomorphisms}.
Similarly, any kind of structure on $V$ with structure preserving maps that can be restricted (and co-restricted) to suitable
subsets, defines a certain pseudogroup, which is an odered subgroupoid of $\cB_{loc}(V)$.
If the property is defined by ``local'' properties (such as smoothness), the following pseudogroup property is ensured:
\end{example}

\begin{definition}\label{def:pseudogroup}
A subgroupoid $G=(G_0,G_1)$ of the full pseudogroup $(\cP(V),\cB_{loc}(V))$ is called a {\em pseudogroup of transformations of $V$}
if:
\begin{enumerate}
\item[{\rm (PsG)}]
assume that
 $U= \cup_{\alpha \in J} U_\alpha$ with all $U_\alpha \in G_0$ and that $f:U \to U' \subset V$ is a bijection such that
$\forall \alpha \in J$: $f\vert_{U_\alpha} \in G_1$.  Then  $f \in G_1$.
\end{enumerate}
\end{definition}

\section{Natural relations between groupoids}\label{app:NR}

Assume $G=(G_0,G_1)$ and $G'=(G_0',G_1')$ are groupoids. 
A {\em morphism} between $G$ and $G'$, or {\em functor}, is a pair of maps $h_0:G_0 \to G_0'$,
$h_1:G_1 \to G_1'$ preserving all structures.
However,  there are other ways to turn groupoids into a category, such as the following: 


\begin{definition}\label{def:naturalrelation}
A {\em natural relation between $G$ and $G'$}  is given by a pair $(F,K)$ of binary relations on objects, and a binary relation $R$ on
morphisms 
$$
(F,K) \in \cP(G_0'\times G_0)^2, \qquad 
R \in   \cP(G_1' \times G_1)
$$
such that:
\begin{enumerate}
\item[{\rm (NR)}]
if $(g,h) \in R$, then letting  $x:=\pi_0(g)$, $x':=\pi_0(g')$, $y:=\pi_1(g)$, $y':=\pi_1(h)$,
we have $(x,x') \in F$ and $(y,y') \in K$: 
$$
\xymatrix{
 x  \ar@{-}[rd]^F \ar[rr]^g  & & y \ar@{-}[rd]^K &  \\
 &  x' \ar[rr]^h & & y'
 }$$
\end{enumerate}
If, moreover, $F=K$, we speak of a {\em special natural relation}.
\end{definition}

\begin{remark} Condition
(NR) implies that, under relational composition,  
$$
\pi_0 \circ R = F \circ \pi_0, \qquad \pi_1 \circ R = K \circ \pi_1.
$$  
\end{remark}

\begin{example}
In an ordered groupoid, $L$ is a special natural endorelation of $G$. If we wanted to draw diagrams of 
natural relations between
ordered groupoids, we would need a third dimension for representing them!
\end{example}

It should be clear that composition of natural relations gives again a natural relation: if
$(F,K;R),(F',K';R')$ are natural between $G$ and $G'$, resp.\ between $G'$ and $G''$, then 
$(F' \circ F,K' \circ K;R' \circ R)$ is natural between $G$ and $G''$. 
For the purposes of the present paper, this remark suffices. 
However, it is certainly useful to note that
just as for natural transformations in general category theory, there are in fact {\em two} compositions, and that we get a
{\em double category} (see, e.g., \cite{Be15a, Be15b} for definitions):

\begin{theorem}
The set $\cN(G)$ of natural relations on a groupoid $G$ forms (the set of $2$-morphisms of) a small double category
$$
\xymatrix{ \cN(G) \ar[r]  \ar@<1ex>[r]  \ar[d]  \ar@<1ex>[d] & \cP(G_1) \ar[d]  \ar@<1ex>[d]  \\
\cP(G_0 \times G_0) \ar[r]  \ar@<1ex>[r]  & \cP(G_0)}
$$
where the two compositions on $2$-morphisms are given by:
\begin{enumerate}
\item
relational composition:
if $(F,K;R)$ and $(F',K';R')$ are composable pairs of natural relations, then
$(F' \circ F, K' \circ K;R' \circ R)$ is again a natural relation,
\item
pointwise $\ast$-composition:
if $(F,K;R)$ and $(F',K';R')$ are natural relations on the same groupoid $G$ and $K=F'$, let $R' \ast R := $
$$ \qquad \quad
\Bigsetof{ (g'',h'') \in G_1 \times G_1} {
\begin{array}{c}
\exists (g',h') \in R', (g,h) \in R , \\ (g',g) , (h',h) \in G_1 \times_{G_0} G_1, \, 
g'' = g' \ast g, h'' = h' \ast h \end{array}  }\ .
$$
Then $(F,K';R' \ast R)$ is again a natural relation.
\end{enumerate}
\end{theorem}

\begin{proof}
The shortest proof is probably by remarking first that the {\em pair groupoid functor} $\pG$  applied to a  small category $C=(C_0,C_1)$,  yields a small double cat $\pG(C)$: 
$$
\xymatrix{  C_1 \times C_1  \ar[r]  \ar@<1ex>[r]  \ar[d]  \ar@<1ex>[d] & C_1 \ar[d]  \ar@<1ex>[d]  \\
C_0 \times C_0  \ar[r]  \ar@<1ex>[r]  & C_0.}
$$
(Indeed, $\pG$ is a {\em cat rule} in the sense of Appendix B in \cite{Be15b}.)
Applying the power set functor $\cP$ to this, we get again a small double cat $\cP(\pG (C))$. Now observe that the structure defined in the theorem,
when $G=C$, is  nothing but a sub-double cat of this small cat.
Indeed, composition of $2$-morphisms corresponds to putting parallelograms as depicted in Def.\ \ref{def:naturalrelation}, side by side, resp.\
one over the other. 
\end{proof}

It is obvious that the special natural endorelations of $G$ form a sub-doublecat, where the upper horizontal double arrows are replaced
by a single arrow. 
 On the other hand, to every natural endorelation $(F,K;R)$ we may associate a relation $\check R$ between objects and morphisms of $C$
 by restricing to units:
$$
\check R := \{ (x,g) \in G_0 \times G_1 \mid \, (\delta(x),g) \in R \},
$$
and define $\check \cN(G) \subset \cP(G_0 \times G_1) = \{ \check R \mid R \in \cN(G) \}$.
The relation $\check R$ is the precise analog of what one might call a
{\em natural transformation from the functor $F:C_0 \to C_0'$ to the functor $K:C_0 \to C_0'$}.
Just as natural transformations form a (strict) $2$-category, so does $\check \cN(G)$:
$$
\xymatrix{  \check \cN(G)   \ar[r]  \ar@<1ex>[r]  \ar[rd]  \ar@<1ex>[rd] &  \cP(G_1)  \ar[d]  \ar@<1ex>[d]  \\
&  \cP(G_0).}
$$

\begin{remark}\label{rk:NR}
It seems that the concept of natural relation  has so far not  yet been considered in category theory (the reason might be
that people generally work in the context of ``large'' categories, where one is not used to consider binary relations that are not
necesarily functional).  
See, however,  a short remark in the $n$-lab,  \url{https://ncatlab.org/nlab/show/natural+transformation}
(``An alternative but ultimately equivalent way...''), pointing into the direction pursued here.
\end{remark}

\end{document}